\documentclass[11pt]{amsart}
\usepackage{amssymb,amscd,amsthm,amsmath,color,hyperref}
\usepackage{amsmath}
\usepackage{amssymb}
\usepackage{epsfig}
\usepackage{graphicx}
\usepackage[frame,ps,matrix,arrow,curve,rotate]{xy}
\usepackage{color}
\definecolor{shadecolor}{gray}{0.875}
\usepackage{amscd}

\numberwithin{equation}{section}

\newcommand{\PP}{\mathbb{P}}
\newcommand{\OO}{\mathcal{O}}

\newcommand{\MM}{\mathcal{M}}

\newcommand{\CC}{\mathbb{C}}

\newcommand{\Hom}{\operatorname{Hom}}
\newcommand{\Homn}{\operatorname{Hom^0}}

\newcommand{\codim}{\operatorname{codim}}

\newcommand{\PGL}{\operatorname{PGL}}
\newcommand{\xX}{\mathcal{X}}

\newcommand{\cU}{\mathcal{U}}
\newcommand{\cB}{\mathcal{B}}

\newcommand{\cX}{\mathcal{X}}
\newcommand{\cY}{\mathcal{Y}}

\theoremstyle{plain}
\newtheorem{lemma}{Lemma}[section]
\newtheorem*{theorem*}{Theorem}
\newtheorem*{lemma*}{Lemma}
\newtheorem*{proposition*}{Proposition}
\newtheorem*{conjecture*}{Conjecture}
\newtheorem*{corollary*}{Corollary}
\newtheorem*{problem*}{Problem}
\newtheorem{theorem}[lemma]{Theorem}
\newtheorem{conjecture}[lemma]{Conjecture}
\newtheorem{corollary}[lemma]{Corollary}
\newtheorem{proposition}[lemma]{Proposition}

\newtheorem{question}[lemma]{Question}

\theoremstyle{definition}

\newtheorem{remark}[lemma]{Remark}

\begin{document}

\title{Restrictions on rational surfaces lying in very general hypersurfaces}
\author[R. Beheshti]{Roya Beheshti}
\address{Department of Mathematics and Statistics\\Washington University in St. Louis, MO 63105}
\email{beheshti@wustl.edu}

\author[E. Riedl]{Eric Riedl}
\address{Department of Mathematics\\University of Notre Dame, IN 46556}
\email{ebriedl@nd.edu}

\begin{abstract}
We study rational surfaces on very general Fano hypersurfaces in $\mathbb{P}^n$, with an eye toward unirationality. We prove that given any fixed family of rational surfaces, a very general hypersurface of degree $d$ sufficiently close to $n$ and $n$ sufficiently large will admit no maps from surfaces in that family. In particular, this shows that for such hypersurfaces, any rational curve in the space of rational curves must meet the boundary. We also prove that for any fixed ratio $\alpha$, a very general hypersurface in $\mathbb{P}^n$ of degree $d$ sufficiently close to $n$ will admit no generically finite maps from a surface satisfying $H^2 \geq \alpha HK$, where $H$ is the pullback of the hyperplane class from $\mathbb{P}^n$ and $K$ is the canonical bundle on the surface.
\end{abstract}

\maketitle

\section{Introduction}
There are many competing notions for what it means for a variety to be ``like'' projective space. Three of the most common are: \emph{rational}, meaning birational to $\PP^n$; \emph{unirational}, meaning admitting a dominant morphism from $\PP^n$; and \emph{rationally connected}, meaning for two general points, there exists a rational curve through both. Celebrated results of Griffiths-Harris, Artin-Mumford, and Iskovskikh-Manin \cite{clemens-griffiths, artin-mumford, iskovskikh-manin} show that there are unirational varieties that are not rational. However, it remains an open question whether rationally connected varieties are always unirational.

\begin{question}
\label{q-RCnotUnirat}
Does there exist a variety that is rationally connected but not unirational?
\end{question}

Due to the classification of surfaces, any counterexample would need to have dimension at least three. It is generally expected that the answer to Question \ref{q-RCnotUnirat} is yes, and an often-discussed source of examples is Fano hypersurfaces of large degree. A smooth hypersurface of degree $d$ in $\PP^n$ is Fano if $d\leq n$, and every Fano variety  is rationally connected. It is known that smooth hypersurfaces of degree 
$d$ in $\PP^n$ are unirational when $2^{d!} \leq n$ (see \cite{beheshti-riedl2} and \cite{hmp}), but it is expected that very general hypersurfaces of large degree, specifically those with degree approximately $n$ in $\PP^n$, are not unirational. Schreieder \cite{schreieder} has recently proved that the degree of a unirational parametrization of a Fano hypersurface of large degree must be 
extremely large by showing that it should be divisible by every integer $m$ such that $m \leq d - \log_2 n$.

Since any unirational variety will be swept out by rational surfaces, we can answer Question \ref{q-RCnotUnirat} negatively by finding a rationally connected variety that is not swept out by rational surfaces, as proposed by Koll\'{a}r. To that end, there has been a lot of past work studying rational surfaces on Fano hypersurfaces. Testa \cite{testa} generalizes work of Beheshti and Starr \cite{beheshti-starr} and proves that a smooth complete intersection of index 1 is not swept out by rational surfaces $S$ with $\omega_S$ nef. Observe that a rational curve in the space of rational curves on $X$ corresponds to a rational surface in $X$. Beheshti \cite{beheshti} proves that spaces of rational curves of low degree are not uniruled, and in \cite{beheshti-riedl} the authors generalize that work by showing that there are no rational surfaces in $X$ ruled by low-degree rational curves, the generic one of which is smooth.

In this paper, we prove several results restricting the types of rational surfaces that lie in a general hypersurface. Roughly speaking, we show that given a fixed surface or family of surfaces, a general Fano hypersurface of degree approximately $n$ in $\PP^n$ admits no generically finite maps from these surfaces. More precisely, we prove the following.

\begin{theorem}
\label{thm-introThm}
Let $X$ be a hypersurface of degree $d$ in $\PP_{\CC}^n$. Then: \begin{enumerate}
\item (cf Corollary \ref{cor-noFixedFamily}) If $n \geq d > \frac{(2-\sqrt{2})(3n+k+1)}{2}+2$ and $X$ is very general with respect to some fixed $k$-dimensional family of rational surfaces $\mathcal{S} \to B$, then $X$ admits no generically finite maps from a fiber of $\mathcal{S} \to B$. In particular, $X$ contains no Hirzebruch surfaces, so there is no complete rational curve in the locus parametrizing embedded smooth rational curves in the Kontsevich space $\overline{\MM}_{0,n}(X,e)$. 
\item (cf Corollary \ref{cor-alphaLambdaNoSurfaces}) If $\alpha$ is a fixed positive number and $\lambda < 1$ is fixed with $\lambda > \frac{3}{2}(2 - \sqrt{2})$, then for sufficiently large $n$, a very general hypersurface of degree $d \geq n\lambda$ admits no generically finite morphisms from a rational surface $S$ with $H \cdot K \leq \alpha H^2$, where $H$ is the pullback of the hyperplane class to $S$ and $K$ is the canonical class on $S$.
\end{enumerate}
\end{theorem}

For particular types of surfaces, we prove stronger restrictions. See Corollary \ref{cor-delPezzos} for a statement about del Pezzo surfaces and Corollary \ref{cor-HarbourneHirshowitz} for a statement about blowups of $\PP^2$ at general points. 

The basic idea of the proofs is to study the normal sheaf $N_{f/X}$ of a morphism $f$ from a rational surface to $X$ and understand the Euler characteristic of its twists. A direct calculation shows that $N_{f/X}(kH)$ must have negative Euler characteristic for small positive integers $k$. However, a careful analysis of globally generated sheaves on rational surfaces shows that in fact this Euler characteristic must be positive. In Section \ref{sec-Familiesofsurfaces}, we lay out the core of this technique, working with twists of $N_{f/X}$. This allows us to find $d$ and $n$ so that a very general degree $d$ hypersurface $X$ in $\PP^n$ will not be swept out by rational surfaces with $H^2$ larger than some fixed multiple of $HK$. In Section \ref{sec-Onefixedsurface}, we reduce mod $p$ and apply a similar argument to the restricted tangent bundle $f^* T_X$. We show how techniques from \cite{CoskunRiedl} then imply the statement of Theorem \ref{thm-introThm}.

\subsection*{Acknowledgements} We gratefully acknowledge helpful conversations with Izzet Coskun, Lena Ji, Eric Larson and John Lesieutre and Andrew Putman. We also gratefully acknowledge many helpful comments from the anonymous referee. Roya Beheshti was supported by 
NSF grant DMS-2101935. Eric Riedl was supported by NSF CAREER grant DMS-1945944.

\section{Families of rational surfaces}
\label{sec-Familiesofsurfaces}

By a family of smooth rational surfaces we mean a smooth and projective morphism $q: \mathcal S \to B$ such that $B$ is a quasi-projective variety and 
the fibers of $q$ are rational surfaces.  Let $X$ be a smooth  hypersurface of degree $d\leq n$ in $\PP^n$. Let $S$ be a fiber of $q$ and $f: S \to X$ a 
generically finite morphism. We denote by $N_f$ the normal sheaf of $f$, that is the cokernel of the injective map 
$T_S \to f^*T_X.$ To emphasize the range, we sometimes write $N_{f,X}$ instead of $N_f$.

We start with a positivity result about the normal sheaf of $f$ when $X$ is a very general hypersurfaces. The result and technique come from work of Voisin \cite{Voisin} and Pacienza \cite{Pacienza}, although we give a proof for completeness.

\begin{proposition}\label{cor-hypersurface}
Let $\mathcal{X} \to H^0(\OO_{\PP^n}(d))$ be the universal hypersurface on $\PP^n$ and let $q: \mathcal{S} \to B$ be a family of rational surfaces with  morphisms $\phi:\mathcal{S} \to \mathcal{X}$ and $g: B \to H^0(\OO_{\PP^n}(d))$ commuting with the natural projection maps. Assume that $\phi$ is dominant and its restriction 
$f$ to a general fiber $S := \mathcal S_b$ of $q$ is generically finite. Then $N_f$ is generically globally generated and  $N_f(H)$ is 
globally generated, where $H$ is the pullback of the hyperplane section under $f$. 
\end{proposition} 

\begin{proof} 
 Let $f: S \to X$ be the restriction of $\phi$ to a general fiber of $q$. The fact that $N_f$ is generically globally generated follows from basic deformation theory and the fact that $X$ is swept out by images of surfaces from this family. Indeed, let $B_X=g^{-1}([X])$ and 
 $\phi_X: \mathcal S_X\to X$ the restriction of $\phi$ to the fiber over 
 $B_X$. Then $\phi_X$ is dominant by our assumption. So by generic smoothness, for a general point $(b,s)$ of $\mathcal S_X$, the induced map on Zariski tangent spaces $T_{\mathcal S_X,(b,s)} \to T_{X,x}=f^*T_X|_s$ is surjective. Therefore the map 
$T_{\mathcal S_X,(b,s)} \to N_{f}|_s$ is surjective as well. 
There is a map from $T_{\mathcal S_X,(b,s)}$ to $H^0(N_f)$ (see for example \cite[Theorem 3.4.8]{Sernesi}),
and the map $T_{\mathcal S_X,(b,s)}\to N_{f}|_s$ factors through the map $T_{\mathcal S_X,(b,s)} \to 
H^0(N_f)$, so the desired result follows.

For the second claim, we only use the assumption that  for a general hypersurface $X$, there exists a generically finite morphism from a fiber of $q$ to $X$. 

We can take an \'etale base change $U \to H^0(\OO_{\PP^n}(d))$ to obtain a family $\cX_U \to U$ of hypersurfaces with a family $\cY \to U$ of rational surfaces parametrized by $q$ mapping to fibers of $\cX_U \to U$ via a map $\psi: \cY \to \cX_U$ such that $\cY$ admits a natural $\PGL_{n+1}$ action. Denote by $\pi$ the projection map 
from $\mathcal X_U$ to $\PP^n$ and let $\pi' = \pi \circ \psi$.  The induced map on tangent bundles $T_{\cY} \to \pi'^{*}T_{\PP^n}$ is surjective because of the $\PGL_{n+1}$ invariance of $\cY$. We have the following commutative diagram: 

\catcode`\@=11
\newdimen\cdsep
\cdsep=3em

\def\cdstrut{\vrule height .25\cdsep width 0pt depth .12\cdsep}
\def\@cdstrut{{\advance\cdsep by 2em\cdstrut}}

\def\arrow#1#2{
  \ifx d#1
    \llap{$\scriptstyle#2$}\left\downarrow\cdstrut\right.\@cdstrut\fi
  \ifx u#1
    \llap{$\scriptstyle#2$}\left\uparrow\cdstrut\right.\@cdstrut\fi
  \ifx r#1
    \mathop{\hbox to \cdsep{\rightarrowfill}}\limits^{#2}\fi
  \ifx l#1
    \mathop{\hbox to \cdsep{\leftarrowfill}}\limits^{#2}\fi
}
\catcode`\@=12

\cdsep=3em
$$
\begin{matrix}
& & 0 & & 0 & &  \cr
& & \arrow{u}{} & & \arrow{u}{} & &  \cr
 &  &  \pi'^{*}T_{\PP^n} & \arrow{r}{=} & \pi'^{*}T_{\PP^n} &   \cr
& & \arrow{u}{} & & \arrow{u}{} & &  \cr
0 & \arrow{r}{} &  T_{\cY}  & \arrow{r}{} & \psi^* T_{\cX_U} & \arrow{r}{} & N_{\psi,\cX_U} & \arrow{r}{} & 0          \cr
& &  \arrow{u}{} & & \arrow{u}{} & & \arrow{u}{=} \cr
0 & \arrow{r}{} & T_{\cY/ \PP^n} & \arrow{r}{} & \psi^*T_{\cX_U/ \PP^n} & \arrow{r}{} & N & \arrow{r}{} & 0 \cr
& & \arrow{u}{} & & \arrow{u}{} & &  \cr
& & 0  & & 0 & &  \cr
\end{matrix}
$$
If $u$ is a general point of $U$ and $f: \mathcal Y_u \to \mathcal X_u$ is the restriction of $\psi$ to the fiber over $u$, then 
$N_{\psi,\mathcal X_U}|_{\mathcal Y_u} = N_{f, \mathcal X_u}$, so to show $N_{f, \mathcal X_u}(H)$ is globally generated, it is enough to show 
$T_{\cX_U/ \PP^n} \otimes \pi^*\OO_{\PP^n}(1)$ is globally generated. Consider the following diagram

$$
\begin{matrix}
& & 0 & & 0 \cr
& & \arrow{u}{} & & \arrow{u}{} \cr
& & \pi^* \OO(d) & \arrow{r}{=} & \pi^* \OO(d) \cr
& & \arrow{u}{} & & \arrow{u}{} \cr
0 & \arrow{r}{} & \OO \otimes S_d  & \arrow{r}{} & \pi^*T_{\PP^n} \oplus \OO \otimes S_d & \arrow{r}{} & \pi^*T_{\PP^n} & \arrow{r}{} & 0          \cr
& &  \arrow{u}{} & & \arrow{u}{} & & \arrow{u}{\cong} \cr
0 & \arrow{r}{} & T_{\xX_U/\PP^n} & \arrow{r}{} & T_{\xX_U} & \arrow{r}{} & \pi^*T_{\PP^n} & \arrow{r}{} & 0 \cr
& & \arrow{u}{} & & \arrow{u}{} \cr
& & 0  & & 0 \cr
\end{matrix}
$$

Using the eight lemma and looking at the first column, it follows that $T_{\cX_U/\PP^n}$ is the kernel of the map $\OO \otimes S_d \to \pi^*\OO(d)$. Such bundles are called Lazarsfeld-Mukai bundles, and so we may say $T_{\cX_U/\PP^n} \cong M_d$, the pullback of the Lazarsfeld-Mukai bundle of $\OO_{\PP^n}(d)$. 

We can similarly define $M_1$ to be the kernel of the natural map $\OO \otimes S_1 \to \pi^*\OO(1)$. The bundle $M_d$ admits a surjection from a direct sum of copies of $M_1$, with maps given by multiplication by a general degree $d-1$ polynomial. It follows that $N_{h,\cX_U}$ admits a surjection from a direct sum of copies of $M_1$. Taking the second wedge power of the sequence
\[ 0 \to M_1 \to \OO \otimes S_1 \to \OO(1) \to 0 \]
we see that $M_1(1)$ is globally generated, and hence $M_d(1)$ is too, so the global generation result follows.

\end{proof}

We will use the above proposition to give a lower bound on the Euler characteristic of the twists of the normal sheaf of $f$. To do so, we will need the following result. 
\begin{proposition}\label{vanishing-cohomology}
Let $S$ be a rational surface over an algebraically closed field and $\pi: S \to \PP^1$ a dominant morphism whose general fibers are isomorphic to $\PP^1$. Let $E$ be a coherent sheaf of rank $m$ on $S$ 
which is generically globally generated. Assume $A$ and $D$ are divisors on $S$ such that $A$ is nef and big, and 
\begin{itemize}
\item[(a)] $E(A)$ is globally generated,
\item[(b)] $ \deg D|_C \geq \deg E|_C$ for a general fiber $C$ of $\pi$,
\item[(c)] $H^1(\OO_S(D))=0$.
\end{itemize}
Then 
$\chi (E(A+D)) \geq  m \; \chi (\OO_S(A+D)).$
\end{proposition}

\begin{proof}
Since $E$ is generically globally generated, choosing $m$ general sections of $E$, we get an injective map $\OO_S^m \to E$. The cokernel of this map, denoted by $T$,  is a torsion sheaf. Since $E(A)$ is globally generated, there is a surjective map $\OO_S^l \to E(A)$ for some $l$. This in turn gives a surjective map $\OO_S^l \to T(A)$ whose kernel we denote by 
$M$. 
$$ \xymatrix{
& & 0 \ar[d] & 0 \ar[d] & \\  
& &    \OO_S(A)^{m} \ar[d] \ar[r]^{=} & \OO_S(A)^{m}  \ar[d]  & \\
0 \ar[r] & M \ar[r]  \ar[d]^{=} & \OO_S^{ l} \oplus \OO_S(A)^{m} \ar[r] \ar[d] & E(A)  \ar[r] \ar[d] & 0  \\
0 \ar[r] & M \ar[r] & \OO_S^{l} \ar[r] \ar[d] & T(A) \ar[d] \ar[r] & 0 \\
& & 0 & 0. & 
}
$$
We claim $H^1( T(A+D)) = 0$. By our assumption $H^1(\OO_S(D)) =0$, so applying the long exact sequence of cohomology to the first sequence twisted with $\OO_S(D)$, the claim follows if we show $H^2(M(D))) = 0$. Applying the Leray spectral sequence corresponding to the map $p$, it is enough to show that $H^1(M(D)|_C)=0$ where $C$ is a  general fiber of $\pi$. Since $C$ is a general fiber, the first short exact sequence above 
remains exacts after restricting to $C$, so $M|_C$ is torsion free. If $M|_C = \OO(a_1) \oplus \dots \oplus \OO(a_l)$, then since $M|_C$ injects into $\OO_C^{ l}$, 
we have $a_i \leq 0$ for each $i$. Also, $$\sum_{i=1}^l a_i = - \deg (T(A))|_C = -\deg T|_C = -\deg E|_C$$ since $T|_C$ is torsion. 
If $a_i < -1-D \cdot C$ for some $i$, then $$ 0 \geq \sum_{j \neq i} a_j = - \deg E|_C - a_i > -\deg E|_C+D\cdot C +1,$$ contradicting assumption (b). So 
$a_i \geq -1 - D \cdot C$ for every $i$ and therefore $H^1(M(D)|_C) =0$. 

This shows that $H^1(T(A+D))=0$, so $\chi (T(A+D)) \geq 0$, and the desired result follows from the short exact sequence 
$$ 0 \to \OO_S(A+D)^m \to E(A+D) \to T(A+D) \to 0.$$
\end{proof}

\begin{corollary} 
\label{cor-eulerChar}
With the same assumptions as in Proposition \ref{cor-hypersurface}, 
$$\chi (N_f((n+2-d)H+K) )\geq (n-3) (\frac{(n+2-d)^2H^2 +(n+2-d) H \cdot K}{2}+1).$$
\end{corollary} 
\begin{proof}
By Proposition \ref{cor-hypersurface}, for a general $f: S \to X$, $N_f$ is generically globally generated and $N_f(H)$ is globally generated. We apply Proposition \ref{vanishing-cohomology} to $E=N_f$, $D=(n+1-d)H+K$, and $A=H$. After possibly blowing up $S$ at a point, we may assume there is a morphism $\pi: S \to \PP^1$ whose general fibers are smooth rational curves. Condition (a) of Proposition \ref{vanishing-cohomology} is satisfied by our assumption. Since $\deg N_f|_C = (n+1-d) H \cdot C - 2$ for a general fiber $C$ of $\pi$, condition (b) is satisfied. By the Kawamata-Viehweg vanishing theorem, condition (c) is also satisfied, so $\chi (N_f((n+2-d)H+K)) \geq  (n-3) \; \chi (\OO_S(n+2-d)H+K).$
 Applying the Riemann-Roch theorem we get the desired inequality. 

\end{proof}

In the next lemma, we calculate the Euler characteristics of the normal sheaf of $f: S \to X$ twisted with $tH+K$ directly. 

\begin{lemma}\label{euler-char}
Let $X$ be a smooth hypersurface of degree $d$ in $\PP^n$, and let $f:S \to X$ be a morphism from a smooth rational surface $S$. If $H$ denotes the pull-back of the 
hyperplane section, $K$ the canonical divisor of $S$, and $N_f$ the normal sheaf of $f$, then we have 
\begin{equation*}
\begin{split}
\chi(N_f(tH+K)) & = \frac{1}{2} ((n-3)t^2+2t(n+1-d)+n+1-d^2) \; H^2\\
& + \frac{1}{2} (t(n-1)+n+1-d) \; H \cdot K \\
& - K^2 +n+9.
\end{split}
\end{equation*}

\end{lemma}
\begin{proof}
This is a straightforward computation using the pullback of the Euler sequence on $\PP^n$ and twisting it with $tH+K$: 
$$ 0 \to \OO_S(tH+K) \to \OO((t+1)H+K)^{n+1} \to f^*T_{\PP^n}(tH+K) \to 0$$ and  the short exact sequence
$$ 0 \to f^*T_X(tH+K) \to f^*T_{\PP^n}(tH+K) \to \OO((d+t)H+K) \to 0.$$
\end{proof}

\begin{theorem}\label{euler-inequality}
With the same assumptions as in Proposition \ref{cor-hypersurface}, we have  
\begin{equation}\label{normal-inequality}
(2(n+1-d)(n+2-d)+n+1-d^2) \; H^2 + (3n-3d+5) \; H \cdot K -2K^2+24 \geq 0
\end{equation}
where  $K$ is the canonical divisor on $S$. 
\end{theorem}

\begin{proof}
This follows from comparing Corollary \ref{cor-eulerChar} with Lemma \ref{euler-char} when $t=n+2-d$.
\end{proof}

\begin{corollary} \label{cor-alphaLambda}
Let $\alpha$ be a fixed positive number and $\lambda$ be a number satisfying $1 > \lambda > 2-\sqrt{2}$. Then for sufficiently large $n$, a very general hypersurface of degree $d \geq \lambda n$ is not swept out by images of generically finite morphisms from a rational surface $S$ with $H \cdot K \leq \alpha H^2$ on $S$. 
\end{corollary}

To prove the corollary, we use Reider's theorem \cite{Reider}. 

\begin{theorem}{(Reider)} 
Let $X$ be a smooth projective surface and $L$ a nef divisor on $X$
with $L^2 \geq 5$. If $|L+K_X|$ has a base-point $x \in X$, then there is an effective divisor $D$ containing $x$ satisfying 
$$ L \cdot D =0, \;\;\; D^2=-1$$ 
or 
$$ L \cdot D=1, \;\;\;\; D^2=0.$$
\end{theorem}

\begin{proof}[Proof of Corollary \ref{cor-alphaLambda}] 
Suppose to the contrary that a very general hypersurface $X$ is covered by images of such morphisms. Blowing down $S$, we can further assume $f: S \to X$ does not contract any $(-1)$-curve.  
Applying Theorem \ref{euler-inequality} to the hypotheses given we conclude that there is a generically finite morphism $f: S \to X$ which does not contract any $(-1)$-curve and satisfies the inequality of Theorem \ref{euler-inequality}. Since $H$ is nef, applying Reider's theorem to $3H$ we see that $3H+K$ is base-point free. Therefore $(3H+K)^2\geq 0$, so $-2K^2 \leq 18H^2+12H \cdot K$.  
Since $d \geq \lambda b$ and $\lambda > 2-\sqrt{2}$, the coefficient of $H^2$ in \ref{normal-inequality} becomes arbitrarily negative compared to the coefficient of $H \cdot K$, so we get a contradiction.
\end{proof}

\begin{remark}
Unfortunately, there exist examples of rational surfaces containing divisors $H$ with $H^2$ small relative to $H \cdot K$.

Take a general pencil of degree $b$ curves in $\PP^2$, and let $S$ be the blowup of $\PP^2$ along the $b^2$ base points of the pencil. Let $H = (b+1)L - \sum_i E_i$, where the sum ranges over all of the exceptional divisors. Then $H$ is base-point free and big. Moreover, $H^2 = (b+1)^2 - b^2 = 2b+1$, while $H \cdot K = -3b + b^2 = b(b-3)$. Thus, $H \cdot K$ grows faster than $H^2$ as $b$ becomes large, and so we cannot hope to obtain a linear bound for $H \cdot K$ in terms of $H^2$.

For another example, consider Example 2 from \cite{CKLLMMT}. The authors describe a blowup $S$ of $\PP^2$ at 19 points together with a sequence of big and nef divisors $D_n$ such that $K_S \cdot D_n$ goes to infinity while $D_n^2 = 2$ for all $n$.
\end{remark}

\begin{corollary} \label{cor-delPezzos} If $X$ is a very general hypersurface in $\PP_{\CC}^n$ of degree $d > (2-\sqrt{2})n+3$,  then the images of generically finite morphisms from del Pezzo surfaces to $X$ cannot sweep out $X$.
\end{corollary}

We remark that $2-\sqrt{2} \approx .59$, so the result holds for $d > \frac{3n}{5}+3$.

\begin{proof}
Suppose to the contrary that $X$ is covered by images of generically finite morphisms from del Pezzo surfaces. Then by Theorem \ref{euler-inequality} 
there exist a del Pezzo surface $S$ and a generically finite morphism $f: S \to X$  for which inequality (\ref{normal-inequality}) holds.  We show this is not possible. Let $B$ be the coefficient of $H^2$ in inequality  (\ref{normal-inequality}). We first show $B < -5n-4$. To see this note that since $d > (2-\sqrt{2})n+3$, $n-d < (\sqrt{2}-1)n-3$, so 
$$B < 2((\sqrt{2}-1)n-2)((\sqrt{2}-1)n-1)+n+1-((2-\sqrt{2})n+3)^2 \leq  -5n-4.$$ Since $S$ is a del Pezzo surface, $-K$ is effective, so $H \cdot K < 0$,
and $K^2 > 0$. So the left hand side of inequality (\ref{normal-inequality}) is at most 
$-5n-4-(3n-3d+5)-2+24$  which is negative since $d \leq n$ and $n \geq 3$. This gives a contradiction.
\end{proof}

Next we apply Theorem \ref{euler-inequality} to the blow-up of $\PP^2$ in general points. Recall the following conjecture of Harbourne 
and Hirschowitz (\cite{Harbourne} and \cite{Hirschowitz}).

\begin{conjecture}{(Harbourne-Hirschowitz)}
Let $S$ be the blow-up of $\PP^2$ at $k$ general points and $L$ a line bundle on $S$. Then $h^1(L) \neq 0$ if and only if there is a $(-1)$-curve 
$E$ in $S$ such that 
$$\deg(L|_E) \leq -2.$$
\end{conjecture}

It is known that the Harbourne-Hirschowitz conjecture holds for $m\leq 9$ \cite[Theorem 5.1]{Cilberto}.

\begin{corollary}
\label{cor-HarbourneHirshowitz}
Suppose $d > (2-\sqrt{2})n+4$, $n \geq 4$, and $X$ is a very general hypersurface of degree $d$ in $\PP_{\CC}^n$.  If the Harbourne-Hirschowitz Conjecture holds true, then the images of generically finite morphisms from blow-ups of $\PP^2$ in general points do not cover $X$.
\end{corollary}
\begin{proof}
Suppose to the contrary that a very general hypersurface $X$ of degree $d$ is covered by the images of generically finite morphisms from blow-ups of $\PP^2$ in general points. 
Then by Theorem \ref{euler-inequality}, there is a rational surface $S$ obtained by blowing up $\PP^2$ in general points and a generically finite morphism $f: S \to X$ such that inequality (\ref{normal-inequality}) is satisfied. 
We can assume $f$ does not contract any $(-1)$-curve since otherwise we can consider the induced morphism from the blow-down of $S$ to $X$ instead.
By the Harbourne-Hirschowitz conjecture, $H^1(S,\OO_S(H))=0$, so by the Riemann-Roch theorem, $H(H-K) =2\chi(H) -2\geq -2$. Therefore $H \cdot K \leq H^2+2$. 

By \cite[Theorem 2.3]{RY} if $n < \frac{d^2+3d+6}{6}$, then there is no rational curve on the space of lines in $X$. Our assumption on the degree implies this inequality is satisfied, so we can assume the image of $S$ under $f$ is not covered by
a pencil of lines. In particular, the image of $S$ has degree at least $2$ and so $H^2 \geq 2$. Since the image of $S$ is not covered by lines, and since $H$ is big and nef, by Reider's theorem \cite{Reider}, $2H+K$ is base-point free. Therefore $(2H+K)^2\geq 0$, so $-2K^2 \leq 8H^2+8H \cdot K \leq 16H^2+16$. 
So if we denote the left hand side of inequality \ref{normal-inequality} by $A$, then we have 
\begin{equation*}
\begin{split}
 A & \leq (2(n+1-d)(n+2-d)+n+1-d^2+(3n-3d+5)+16) \; H^2 + 2(3n-3d+5)  + 40. \\
 \end{split}
 \end{equation*}
 Denote the coefficient of $H^2$ in the above inequality by $B$. We claim $B < -(3n-3d+5)-20$. This is because by our assumption $n-d \leq (\sqrt{2}-1)n-4$, so 
\begin{equation*}
\begin{split}
 B + (3n-3d+5)+20  &< 2((\sqrt{2}-1)n-3)((\sqrt{2}-1)n-2)+n+1-((2-\sqrt{2})n+4)^2  \\
& \;\;\;\;\; +6((\sqrt{2}-1)n-4) +46 \\
& = (4\sqrt{2}-11)n+19\\
& <0,
 \end{split}
 \end{equation*}
 Where the last line follows from the assumption $n \geq 4$. Since $H^2 \geq 2$, we get $A <0$, a contradiction.

\end{proof}

\section{Morphisms from a fixed rational surface}
\label{sec-Onefixedsurface}
Here we consider maps from a fixed rational surface $S$. Let $S$ be a smooth rational surface, and let $\Hom^0(S,X)$ denote the open locus in $\Hom(S,X)$ parametrizing 
generically finite morphisms from $S$ to $X$. We say that $S$ \emph{strongly sweeps out} a variety $X$ if the natural map $\Hom^0(S,X) \times S \to X \times S$ given by $(f,p) \mapsto (f(p), p)$ is dominant. In other words, $X$ is strongly swept out by $S$ if for any pair of a general points $p \in S$ and $q \in X$, there is a generically finite morphism $S \to X$ sending $p$ to $q$.

\begin{proposition}
\label{lem-TxgenericallyGlobGen}
Suppose that $X$ is a very general hypersurface of degree $d$ in $\PP_{\CC}^n$ and $X$ is strongly swept out by $S$. Then there is $f \in \Hom^0(S,X)$ such that $f^* T_X$ is generically globally generated and $f^*T_X(H)$ is globally generated.
\end{proposition}
\begin{proof}
Consider the map $\Homn(S,X) \times S \to X \times S$. Then for any $(f,p) \in \Hom^0(S,X) \times S$ we have the following diagram.

\catcode`\@=11
\newdimen\cdsep
\cdsep=3em

\def\cdstrut{\vrule height .25\cdsep width 0pt depth .12\cdsep}
\def\@cdstrut{{\advance\cdsep by 2em\cdstrut}}

\def\arrow#1#2{
  \ifx d#1
    \llap{$\scriptstyle#2$}\left\downarrow\cdstrut\right.\@cdstrut\fi
  \ifx u#1
    \llap{$\scriptstyle#2$}\left\uparrow\cdstrut\right.\@cdstrut\fi
  \ifx r#1
    \mathop{\hbox to \cdsep{\rightarrowfill}}\limits^{#2}\fi
  \ifx l#1
    \mathop{\hbox to \cdsep{\leftarrowfill}}\limits^{#2}\fi
}
\catcode`\@=12

\cdsep=3em

$$
\begin{matrix}
 T_{\Homn(S,X) \times S, (p,f)} & \arrow{r}{\alpha} & T_{X \times S, (f(p),p)}  \cr
 \arrow{d}{\pi_1} & & \arrow{d}{\pi_2} \cr
 T_{\Hom^0(S,X),f} = H^0(f^*T_X) & \arrow{r}{\beta} & T_{X,f(p)}
\end{matrix}
$$

Let $(p,f)$ be a general point of an irreducible component of $\Hom^0(S,X) \times S$ which dominates $X \times S$. Then if $\mathcal{U}_{red}$ is the largest reduced subscheme of $\Hom^0(S,X)$, generic smoothness shows that $\alpha|_{T_{\mathcal{U}_{red},(p,f)}}$ is surjective, and hence $\pi_2 \circ \alpha$ is surjective. It follows that $\beta$ is surjective, and hence, $f^*T_X$ is generically globally generated. 

To prove $f^*T_X(H)$ is globally generated, we repeat the argument of Proposition \ref{cor-hypersurface} with $\cX$ replaced by $\cX \times S$. We can find a dominant 
morphism $U \to H^0(\OO_{\PP^n}(d))$ and a morphism $\psi: U \times S \to \cX_U \times S$ such that if $\pi': U \times S \to \PP^n$ is the natural morphism, then the induced map on tangent spaces $T_{U \times S} \to \pi'^* T_{\PP^n\times S}$ is surjective.  The same argument as in Proposition \ref{cor-hypersurface} shows that $T_{\cX_U \times S / \PP^n \times S}(H)$ is globally generated, and therefore $N_{\psi, \cX_U \times S}(H)$ is globally generated. if $f: S \to X$ is the restriction of $\psi$ to $\{u\}\times S$ for a general point $u$ of $U$, then $N_{\psi, \cX_U \times S}|_{u \times S} = N_{f, X \times S} = f^*T_X$, the desired  result follows. 
\end{proof}

\begin{proposition} \label{char-p}
Suppose $S$ is a fixed rational surface.  If $d > (2-\sqrt{2})n+2$ and $X$ is a very general hypersurface of degree $d$ in $\PP_{\CC}^n$, 
then $X$ is not strongly swept out by $S$. 
\end{proposition}
\begin{proof}

Assume to the contrary that $X$ is strongly swept out by $S$. Then by Proposition \ref{lem-TxgenericallyGlobGen}, there is $f:S \to X$ such that 
$f^*T_X$ is generically globally generated and 
$f^*T_X(H)$ is globally generated. Passing to finite characteristic, we can assume there is a sufficiently large prime number $p$, a rational surface 
$S_p$, a smooth hypersurface $X_p$ of degree $d$, and a morphism $f_p:S_p\to X_p$ all defined over an algebraically closed field of characteristic $p$ such that $f_p^*T_{X_p}$ is generically globally generated and 
$f_p^*T_{X_p}(H)$ is globally generated.  

For a coherent sheaf $F$ on $S_p$, let $F^{(p)}$ denote the pullback of $F$ under the absolute Frobenius morphism of $X$ (see 
\cite[Proposition 6.1]{Hartshorne}.) 
Since $f_p^*T_{X_p}(H)$ is globally generated, $(f_p^*T_{X_p})^{(p)}(pH)$ is also globally generated. Since $f_p^*T_{X_p}$ is generically globally generated, $(f_p^*T_{X_p})^{(p)}$ will be as well. 
Let $K$ denote the canonical divisor of $S_p$. A similar computation as in Lemma \ref{euler-char} shows that
\begin{equation}\label{euler-charp}
 \chi(f_p^*T_{X_p}^{(p)}(tpH+K)) = \frac{1}{2}(p^2(n(t+1)^2+2t+1-(d+t)^2) H^2 + p(nt+n+1-d-t)HK) +(n-1).
 \end{equation}

We can now apply Proposition \ref{vanishing-cohomology} with $E=f_p^*T_{X_p}^{(p)}$, $A =pH$,  and $D=p(n+2-d)H+K$ to get a contradiction. Since $S_p$ is a rational surface, by
 \cite[Theorem 3]{Mukai}, $H^1(D+pH)=0$, so condition (c) of Proposition \ref{vanishing-cohomology} is satisfied. After possibly blowing up $S_p$ at a point, we can assume there is a morphism $S_p \to \PP^1$ whose general fibers are smooth rational curves. If $C$ denotes a general fiber of this morphism, then $\deg (f_p^*T_{X_p}^{(p)})|_C = p(n+1-d) H \cdot C  \leq p(n+2-d)H\cdot C -2$, so 
condition (b) is also satisfied. Therefore, 
\begin{equation*}
\begin{split}
 \chi(f_p^*T_{X_p}^{(p)}((n+3-d)pH+K)) & \geq (n-1) \chi (\OO_{S_p}((n+3-d)pH+K))\\
 & = \frac{1}{2}(p^2(n-1)(n+3-d)^2 H^2+p(n-1)(n+3-d)HK)+ (n-1).
 \end{split}
 \end{equation*}
Comparing the coefficients of $H^2$ in this equation and Equation (\ref{euler-charp}) when $t=n+3-d$, and letting $p$ increase, we see if the above inequality holds then we must have 
$$n(t+1)^2+2t+1-(d+t)^2 \geq (n-1)(n+3-d)^2.$$
Letting $t=n+3-d$, this gives
$$n+(n+3-d)^2+2(n+3-d)(n+1)+1-(n+3)^2 \geq 0.$$
But since we assume $d > (2-\sqrt{2})n+2$, we have $n-d < (\sqrt{2}-1)n-2$, so the left hand side of the above inequality is smaller than 
$$ n+((\sqrt{2}-1)n+1)^2+2((\sqrt{2}-1)n+1)(n+1)+1-(n+3)^2 = -5-(6-4\sqrt{2})n < 0,$$
a contradiction. 
\end{proof}

We now implement a technique from \cite{CoskunRiedl} to show that a very general $X$ admits no generically finite morphisms from $S$. We recall some terminology. Let $\cU_{n,d}$ be the space of pairs $(p,X)$ where $X$ is a degree $d$ hypersurface in $\PP^n$ and  $p \in X \subset \PP^n$. Given a subset $\cB_r \subset \cU_{r,d}$, we let the \emph{tower of induced varieties} of $\cB_r$ be defined inductively by $\cB_{j+1} \subset \cU_{j+1,d}$ is the set of pairs $(p,X)$ such that some linear section of the pair $(p,X)$ is in $\cB_j$, $j \geq r$.

In our setting, we are interested in proving that $\cB_n$ has high codimension in $\cU_{n,d}$, from which it will follow that a very general hypersurface will contain no points of $\cB_n$. The tool we use is the following.

\begin{theorem}[Theorem 4.8 from \cite{CoskunRiedl}]
\label{thm-CoskunRiedl}
Let $\cB_r \subset \cU_{r,d}$ be an integral, $PGL_{r+1}$-invariant subvariety, and let $\cB_{n}, n \geq r,$ be the tower of induced subvarieties of $\cB_r$. Then if $\cB_m$ is not dense in $\cU_{m,d}$ for some $m > r$, either: \begin{enumerate}
\item $\codim \cB_n \subset \cU_{n,d}$ is at least $2(m-n)+1$ for every $r \leq n \leq m$, or
\item There is some $\cB_1 \subset \cU_{1,d}$ such that $\cB_n$ is in the closure of the tower of induced subvarieties of $\cB_1$, or
\item $\cB_{m,d}$ is the space of pairs $(p,X)$ with $p$ contained in a line $\ell$ lying in $X$.
\end{enumerate}
\end{theorem}

We wish to show that in our setting, we need only consider case (1). 

\begin{corollary}
\label{cor-towerInducedVars}
Let $\cB_r \subset \cU_{r,d}$ be an integral, $PGL_{r+1}$-invariant subvariety, and let $\cB_{n,d}$ be the tower of induced subvarieties of $\cB_r$ for $n \geq r$. Then if $\cB_m$ is not dense in $\cU_{m,d}$ for some $m \geq d+1 \geq r$, $\codim \cB_n \subset \cU_{n,d}$ is at least $2(m-n)+1$.
\end{corollary}
\begin{proof}
First observe that for $m > d+1$, the space of lines in any degree $d$ hypersurface in $\mathbb{P}^n$ will sweep out $X$, so case (3) of Theorem \ref{thm-CoskunRiedl} cannot occur. 
We next show that case (2) cannot occur. Let $\cB_1 \subset \cU_{1,d}$ and let $\cB_n \subset \cU_{n,d}$ be the tower of induced varieties. It suffices to show $\cB_{d+1}$ is dense in $\cU_{d+1,d}$. If a general element of $\cB_1$ has at least 2 distinct points, it follows from \cite{CoskunRiedl} Proposition 4.10(1) that $\cB_n$ is dense in $\cU_{n,d}$, which contradicts the hypothesis. Thus, it remains to consider the cases where $\cB_1$ consists of $d$-fold points, and show that the tower of induced varieties is dense in $\cU_{m,d}$. To see this, let $(p,V(F)) \in \cU_{n,d}$ be a general point. Expand the equation of $F$ around $p$ to get $F = F_1+\cdots F_d$, where $F_i$ is the $i$th order part of $F$ near $p$. Then the space of lines in $\PP^n$ meeting $V(F)$ to order $d$ near $p$ is $V(F_1,\dots,F_{d-1})$. This is nonempty for $d \leq n-1$, so the result follows. 

\end{proof}

Using this result, we can prove a series of results about non-existence of rational surfaces in a very general hypersurface $X$, using our previous work.

\begin{theorem}
\label{thm-noSurfaces}
Let $d, n$ and $c$ be integers with $d < n+c$. If a very general hypersurface of degree $d$ in $\PP_{\CC}^{n+c}$ is not swept out by rational surfaces from a certain class (e.g., images of generically finite morphisms from Hirzebruch surfaces), then for a very general hypersurface $X$ of degree $d$ in $\PP^n$, surfaces in this class sweep out a subvariety of $X$ of codimension at least $2c+1$. In particular, if $c \geq \frac{n-3}{2}$, then $X$ contains no surfaces in this class.
\end{theorem}
\begin{proof}
Let $\cB_{n,d}$ be the locus in $\cU_{n,d}$ swept out by surfaces in the given class, and let $\cB_{r,d}$ be the tower of induced varieties from $\cB_{n,d}$. Then by assumption, we have that $\cB_{n+c,d}$ has codimension at least $1$ in $\cU_{n+c,d}$. By Corollary \ref{cor-towerInducedVars}, it follows that $\cB_{n,d}$ has codimension at least $2c+1$ in $\cU_{n,d}$ and the first part of the result follows. Since generically finite morphisms from a surface to $X$ must sweep out a subvariety of dimension at least 2, we see that $X$ will admit no such morphisms if $2c+1 \geq n-1-1 = n-2$. The second result follows
\end{proof}

\begin{corollary} \label{cor-alphaLambdaNoSurfaces}
Let $\alpha$ be a fixed positive number and $\lambda < 1$ another real number with $\lambda > \frac{3}{2}(2 - \sqrt{2})$. Then for sufficiently large $n$, a very general hypersurface of degree $d \geq n\lambda$ contains no images of generically finite morphisms from a rational surface $S$ with $H \cdot K \leq \alpha H^2$ on $S$.
\end{corollary}
\begin{proof}
This is a direct application of Theorem \ref{thm-noSurfaces} to the results of Corollary \ref{cor-alphaLambda}, replacing $\lambda$ with $\frac{2\lambda}{3}$.
\end{proof}

\begin{corollary}
\label{cor-notSweptOut}
Let $S$ be a rational surface and $X \subset \PP_{\CC}^n$ a very general hypersurface of degree $>(2-\sqrt{2})(n+1)+2$. Then $X$ is not covered by the images of generically finite morphisms from $S$.
\end{corollary}
\begin{proof}
Observe that the result is immediate for $d > n$, so we may assume $d \leq n$. Let $p \in S$ be a general point, and let $\cB_{n,d}$ be the locus of images of $p$ under a generically finite morphism $S \to X$. Let $\cB_{n+1,d}$ be part of the tower of induced varieties of $\cB_{n,d}$. Then by Proposition \ref{char-p}, $\cB_{n+1,d}$ has codimension at least 1 in $\cU_{n+1,d}$. By Corollary \ref{cor-towerInducedVars}, it follows that $\cB_{n,d}$ has codimension at least 3 in $\cU_{n,d}$. Thus, the locus in $X$ swept out by the images of $p$ under generically finite morphisms from $S$ is codimension at least 3 in $X$. It follows that the images of $S$ under generically finite morphisms sweep out a subvariety of codimension at least 1, as required.
\end{proof}

\begin{corollary} \label{cor-noFixedFamily}
Let $\mathcal{S} \to B$ be a family of rational surfaces of dimension $\dim B = k$ and $n$ and $d$ be integers satisfying $n \geq d > \frac {(2-\sqrt{2})(3n+k+1)}{2} +2$. Then if $X \subset \PP_{\CC}^n$ is a very general hypersurface of degree $d$, $X$ admits no generically finite morphisms from any surface in the fibers of $\mathcal{S} \to B$.
\end{corollary}

Note that $\frac{3(2-\sqrt{2})}{2} < \frac{9}{10}$, so the result holds for $d \geq \frac{9n}{10}+\frac{3k}{10}+3$.

\begin{proof}
As we saw in the proof of Corollary \ref{cor-notSweptOut}, for any fixed surface $S$ in the fibers of $\mathcal{S} \to B$, we can construct $\cB_{r,d}$, the tower of induced varieties on the locus in $\cU_{n,d}$ swept out by generically finite images from $S$. We have already seen that $\cB_{m,d}$ has codimension at least $1$ in $\cU_{m,d}$ for some $m$ satisfying $d > (2-\sqrt{2})(m+1)+2$. Thus, $\cB_{m-c,d}$ will have codimension at least $2c+1$ in $\cU_{m-c,d}$ for $r$ in this range. If we let $\cB'_{m,d}$ be the union of all the $\cB_{m,d}$ over all the different possible $S$, then $\cB'_{m-c,d}$ will have codimension at least $2c+1-k$ in $\cU_{m-c,d}$. If $2c+1-k \geq m-c$, then we see that a very general $X$ of degree $d$ in $\PP^{m-c}$ will admit no generically finite maps from the fibers of $S$. Using $m-c = n$ and rearranging, we see that the result holds as claimed.
\end{proof}

\begin{corollary}
If $X$ is a very general hypersurface of degree $d$ in $\PP^n$ and $n \geq d > \frac{(2-\sqrt{2})(3n+1)}{2}+2$, then $X$ admits no generically finite maps from Hirzebruch surfaces. In particular, any rational curve in the space of rational curves on $X$ has to meet the boundary.
\end{corollary}

\bibliographystyle{plain}

\end{document}